\documentclass[11pt]{amsart}
\usepackage{amssymb}
\usepackage{eurosym}
\usepackage{amsfonts}
\usepackage{amsmath}
\usepackage{enumerate}
\usepackage{color}
\usepackage{cite}

\usepackage{tikz}

\usetikzlibrary{patterns}

\newtheorem{theorem}{Theorem}[section]

\newtheorem{corollary}[theorem]{Corollary}

\newtheorem{definition}[theorem]{Definition}
\newtheorem{example}[theorem]{Example}

\newtheorem{lemma}[theorem]{Lemma}

\newtheorem{proposition}[theorem]{Proposition}
\newtheorem{remark}[theorem]{Remark}

\numberwithin{equation}{section} \textwidth=15.5cm
\begin{document}
\title[Periodic solutions for $p(t)$-Li\'{e}nard equations]{Periodic solutions for $p(t)$-Li\'{e}nard equations with a
singular nonlinearity of attractive type}
\author{Petru JEBELEAN}
\address{Institute for Advanced Environmental Research, West University of Timi\c{s}oara,
Blvd. V. P\^{a}rvan, no. 4, 300223  Timi\c{s}oara, Romania,
E-mail : petru.jebelean@e-uvt.ro}
\author{Jean MAWHIN}
\address{Institut de Recherche en Math\'{e}matique et Physique, Universit\'{e} Catholique de Louvain, Chemin du Cyclotron,
2, 1348 Louvain-la-Neuve, Belgium,
E-mail : jean.mawhin@uclouvain.be}
\author{C\u{a}lin \c{S}ERBAN}
\address{Department of Mathematics,
West University of Timi\c{s}oara,
Blvd. V. P\^{a}rvan, no. 4, 300223  Timi\c{s}oara, Romania, E-mail :
calin.serban@e-uvt.ro}

\begin{abstract}
\noindent We are concerned with the existence of $T$-periodic solutions to an equation of type
$$
\left (|u'(t))|^{p(t)-2} u'(t) \right )'+f(u(t))u'(t)+g(u(t))=h(t)\quad \mbox{ in }[0,T]
$$
where $p:[0,T]\to(1,\infty)$ with $p(0)=p(T)$ and $h$ are continuous on $[0,T]$, $f,g$ are also continuous on $[0,\infty)$, respectively $(0,\infty)$. The mapping $g$ may have an attractive singularity (i.e. $g(x) \to +\infty$ as $x\to 0+$). Our approach relies on a continuation theorem obtained in the recent paper \cite{[HMMT]}, a priori estimates and method of lower and upper solutions.
\end{abstract}

\maketitle


\bigskip

\section{Introduction}
We deal with the solvability of $p(t)$-Laplacian Li\'{e}nard periodic boundary value problems of type
\begin{equation}\label{pb1}
	\left \{
            \begin{array}{ll}
		      (\phi_{p(t)}(u'))'+f(u)u'+g(u)=h(t)\quad \mbox{ in }[0,T],\\
		      u(0)-u(T)=0=u'(0)-u'(T),
			\end{array}
\right.
\end{equation}
where $p:[0,T]\to(1,\infty)$ is continuous, with $p(0)=p(T)$ and
$$\phi_{p(t)}(x)=|x|^{p(t)-2}x \qquad (t\in [0,T], \; x\in \mathbb{R}).$$

The functions $f:[0,\infty)\to\mathbb{R}$, $g:(0,\infty)\to\mathbb{R}$ are continuous and $h\in C:=C([0,T]; \mathbb{R})$. Setting $C^1:= C^1 ([0,T]; \mathbb{R})$ and
$$C_T^1:=\{v\in C^1\ |\ v(0)-v(T)=0=v'(0)-v'(T)\},$$
a {\sl solution} of \eqref{pb1} is defined as being a function $u:[0,T]\to(0,\infty)$ with $u\in C_T^1$ and  $\phi_{p(\cdot )}(u'( \cdot ))\in C^1$, which satisfies the differential equation in \eqref{pb1} everywhere on $[0,T]$.
\medskip

As $g$ is only defined on $(0,\infty)$, the emphasis is on the case where $g$ has a singularity at $0$. When
$$\lim_{x \to 0+} g(x) = + \infty, \; \lim_{x \to + \infty} g(x) = 0,$$
the restoring force $g$ is called {\it attractive}, and when
$$\lim_{x \to 0+} g(x) = - \infty, \; \lim_{x \to + \infty} g(x) = 0,$$
 $g$ is called {\it repulsive}.
\medskip

In the classical case where $p(t) = 2$ for all $t \in [0,T]$, the history of such problems has been given in \cite{[Maw93],[JMANS]}.
The first existence results for the problem \eqref{pb1} when $p(t)\equiv p > 1$ is constant seems to have been given in \cite{[JMANS]}, where it was proved  (with $\overline h$ the mean value of $h \in C$) that, if
 \begin{itemize}
 \item[$(A)$] $\displaystyle \liminf_{x \to 0+}[g(x) - \|h\|_\infty] > 0,\; \limsup_{x \to + \infty}[g(x) - \overline h] < 0,$
 \end{itemize}
or if
 \begin{itemize}
\item[$(R)$] $\displaystyle \limsup_{x \to 0+}[g(x) + \min\{0,\overline h\} < 0, \; \liminf_{x \to +\infty} [g(x) + \overline h] >0$,\\
$\displaystyle \int_0^1 g(x)\,dx = - \infty, \; \limsup_{x \to + \infty} g(x) < +\infty,$
\end{itemize}
the problem
\begin{equation}\label{p}
	\left \{
            \begin{array}{ll}
		      (\phi_{p}(u'))'+f(u)u'+g(u)=h(t)\quad \mbox{ in }[0,T],\\
		      u(0)-u(T)=0=u'(0)-u'(T)
			\end{array}
\right.
\end{equation}
has a positive solution. This easily implies that when
$$\lim_{x \to 0+}g(x) = + \infty, \; \lim_{x \to +\infty} g(x) = 0,$$
the problem \eqref{p} has a positive solution if and only if $\overline h > 0$, while when
 $$\lim_{x \to 0+}g(x) = - \infty, \; \lim_{x \to + \infty} g(x) = 0,\; \int_0^1 g(x)\,dx = - \infty,$$
\eqref{p} has a positive solution if and only if $\overline h < 0$.
\medskip

Further refinements, generalizations and variants are due, in chronological order, to I. Rach\r{u}nkov\'{a}, M. Tvrd\'{y} and I.Vrko\v{c} \cite{[RTV]}, I. Rach\r{u}nkov\'{a} and M. Tvrd\'{y} \cite{[RaT1],[RaT2],[RaT3]}, A. Cabada, A. Lom\-ta\-ti\-dze and M. Tvrd\'y \cite{[CLT]}, Y. Xin, X. Han and Z. Cheng \cite{[XHC]}, Y. Xin and Z. Cheng \cite{[XiC1],[XiC2]}, D. Lan and W. Chen \cite{[LaC]}, Y. Xin and H. Liu \cite{[XiL]}, T. Zhou, B. Du and H. Du \cite{[ZDD]}, F. Zheng \cite{[Zhe]}, Y. Xin and Z. Cheng \cite{[XiC3]}, and Y. Zhu \cite{[Zhu]}.
\medskip

The continuation theorem in \cite{[MaM]} for second order systems involving generalizations of the vector $p$-Laplacian used in \cite{[JMANS]} has been recently extended in \cite{[HMMT]} to a class of mappings $\phi : [0,T] \times \mathbb R^N \to \mathbb R^N$ containing in particular, when $N=1$, the mapping $(t,x) \mapsto \phi_{p(t)}(x)$ which engenders the $p(t)$-Laplacian.
The special case of this extension here employed is recalled as Theorem \ref{thHMMT} in Section \ref{sectiunea2bis}, and used there to prove the existence and bounds for the $T$-periodic solutions of equations of the form
$$ (\phi_{p(t)}(u'))'+\theta(u)u'-\varepsilon u=e(t),$$
where $\varepsilon > 0$ and $e$ has mean value zero (Theorem \ref{thepsilon}). The existence result is also extended to the case where $\varepsilon = 0$ in Theorem \ref{thoper},
 and a variant, needed later to construct some upper solutions, is given in Proposition \ref{lemaf0}.
\medskip

The fundamental existence theorem of the method of lower and upper solutions is stated and proved in \cite{[MaT2]}
under more general regularity assumptions for second order quasilinear differential equations and boundary conditions containing the periodic ones.
Based only on Theorem \ref{thHMMT}, we provide in Section \ref{sectiunea2}, for problem \eqref{pb1}, a direct proof of this method (Theorem \ref{lemalu}). Then, Theorem \ref{lemalu} is used in final Section \ref{final} to prove the complete extension to the $p(t)$-Laplacian case of the main existence result in \cite{[JMANS]} for the case of attractive forces. In this moment, the corresponding extension for the case of repulsive restoring forces remains open.
\medskip

Notice finally that variational methods have been used in those last years to prove the existence of periodic solutions of some perturbations of the $p(t)$-Laplacian by, chronologically, X.J. Wang and R. Yuan \cite{[WaY]}, L. Zhang and Y. Chen \cite{[ZhC]}, L. Zhang and X.H. Tang \cite{[ZhT]}, and C. Liu and Y. Zhong \cite{[LiZ]}. This approach cannot be used for problem \eqref{pb1}, which has not a variational structure. In contrast to the topological approach used here, they require the use of the more sophisticated Sobolev spaces with variable exponents.

\section{Some auxiliary problems}\label{sectiunea2bis}

Below, by $\| \cdot \| _{\infty}$ we denote the sup-norm on $C$, the spaces $C^1$ and $C_T^1$ will be considered with the norm
$\|v\|_1=\|v\|_\infty+\|v'\|_\infty$, while the usual norm on $L^r:=L^r([0,T]; \mathbb{R})$ will be denoted by $\| \cdot \|_{L^r}$ ($1\leq r \leq \infty$). By $B(0,\rho)$ we denote the open ball of radius $\rho>0$ in $C_T^1$ and by $\overline B(0,\rho)$ its closure. Also, we use the notations $$p_-:=\min_{t\in[0,T]} p(t), \qquad p^+:=\max_{t\in[0,T]} p(t).$$

Consider the general periodic problem
\begin{equation}\label{pbell}
	\left \{
            \begin{array}{ll}
		      (\phi_{p(t)}(u'))'=\ell(t,u,u')\quad \mbox{ in }[0,T],\\
		      u(0)-u(T)=0=u'(0)-u'(T),
			\end{array}
    \right.
\end{equation}
where $\ell:[0,T]\times\mathbb{R}^2\to\mathbb{R}$ is a Carath\'{e}odory function; recall this means:
\smallskip

($i$) for a.e. $t\in [0,T]$, the function $l(t, \cdot, \cdot )$ is continuous;
\smallskip

($ii$) for each $x,y \in \mathbb{R}$ the function $l( \cdot,x,y)$ is measurable;
\smallskip

($iii$) for each $\rho >0$ there is some $\alpha_{\rho} \in L^1$ such that
$$|l(t,x,y)| \leq \alpha_{\rho}(t), \; \mbox{ for a.e. }t\in [0,T] \mbox{ and all }x,y \in \mathbb{R} \mbox{ with }|x|, |y| \leq \rho.$$

\begin{definition}\label{def1}\emph{By a {\it solution} of problem \eqref{pbell} we understand a function $u\in C^1_T$ with $\phi_{p(\cdot )}(u'( \cdot ))$ absolutely continuous, which satisfies the differential equation in \eqref{pbell} a.e. on $[0,T]$.
}
\end{definition}

\begin{remark}\label{remark1} \emph{Note that if $\ell:[0,T]\times\mathbb{R}^2\to\mathbb{R}$ is continuous and $u$ is a solution  of \eqref{pbell} in the sense of Definition \ref{def1}, then $\phi_{p(\cdot )}(u'( \cdot ))\in C^1$ and $u$ satisfies the differential equation in \eqref{pbell} everywhere on $[0,T]$ -- this actually means that $u$ is solution in classical sense.}
\end{remark}
\smallskip

Below in this section, we adopt a strategy similar to the one from \cite[Section 6]{[Maw01]} for $p(t)$ a constant $p$. We make use of the following continuation theorem, which is an immediate consequence of \cite[Theorem 4.1]{[HMMT]}.
\begin{theorem}\label{thHMMT}
Let $\Omega$ be an bounded open set in $C_T^1$ such that the following conditions hold.
\begin{enumerate}
  \item[$(H_1)$] For each $\lambda\in(0,1)$ the problem
  \begin{equation}\label{pblambda}
	\left \{
            \begin{array}{ll}
		      (\phi_{p(t)}(u'))'=\lambda\ell(t,u,u')\quad \mbox{ in }[0,T],\\
		      u(0)-u(T)=0=u'(0)-u'(T)
			\end{array}
  \right.
\end{equation}
has no solution on $\partial\Omega$.
  \item[$(H_2)$] The equation
  $$\mathcal{L}(a):=\frac{1}{T}\int_0^T\ell(t,a,0)dt=0$$
  has no solution on $\partial\Omega\cap\mathbb{R}$.
  \item[$(H_3)$] The Brouwer degree
  $$deg_B[\mathcal{L},\Omega\cap\mathbb{R},0]\neq0.$$
\end{enumerate}
Then problem \eqref{pbell} has a solution in $\overline{\Omega}$.
\end{theorem}
\medskip

It will be sometimes useful to write an element $v\in L^1$  as $v(t)=\overline{v}+\tilde{v}(t),$  with
$$\overline{v}:=\frac{1}{T}\int_0^T v(t)dt,\ \  \left ( \mbox{whence } \int_0^T\tilde{v}(t)dt=0 \right ).$$
Recall, according to the Sobolev inequality, that any absolutely continuous function $v$ satisfies
\begin{equation}\label{sobolev}
    \|\tilde{v}\|_\infty\leq\|v'\|_{L^1}.
\end{equation}
\smallskip

We first consider the problem
\begin{equation}\label{pbepsilon}
	\left \{
            \begin{array}{ll}
		      (\phi_{p(t)}(u'))'+\theta(u)u'-\varepsilon u=e(t)\quad \mbox{ in }[0,T],\\
		      u(0)-u(T)=0=u'(0)-u'(T),
			\end{array}
\right.
\end{equation}
where $\varepsilon>0$ is a parameter, $\theta:\mathbb{R}\to\mathbb{R}$ is continuous and $e\in L^1$ is with $\overline{e}=0.$

\begin{theorem}\label{thepsilon}
For any $\varepsilon^*>0$ there exists $R:=R(\varepsilon^*)>0$ such that, for all $\varepsilon\in(0,\varepsilon^*]$, problem \eqref{pbepsilon} has at least one solution in $B(0,R)$ and any solution $u$ of \eqref{pbepsilon} belongs to $B(0,R)$ and satisfies $\overline{u}=0$.
\end{theorem}

\noindent\begin{proof} Let $\varepsilon^*>0$ be fixed and $\varepsilon\in(0,\varepsilon^*]$ be arbitrarily chosen. Set $$\Theta(x):=\int_0^x\int_0^s\theta(\sigma)d\sigma ds\quad (x\in \mathbb R)$$ and note that, if $w\in C^1$, then $(\Theta'(w))'=\theta(w)\, w'$. We apply Theorem \ref{thHMMT} with $\ell(t,x,y)=e(t)+\varepsilon x-\theta(x)y$. In this view, let us consider the family of problems
\begin{equation}\label{pbepsilon1}
	\left \{
            \begin{array}{ll}
		      (\phi_{p(t)}(u'))'+\lambda(\Theta'(u))'-\lambda\varepsilon u=\lambda e(t)\quad \mbox{ in }[0,T],\\
		      u(0)-u(T)=0=u'(0)-u'(T)
			\end{array}
\right.
\end{equation}
and let $u(t)=\overline{u}+\tilde{u}(t)$ be a solution of \eqref{pbepsilon1} for some $\lambda\in(0,1]$. Integrating
both members of the equation in \eqref{pbepsilon1} over $[0,T]$, taking into account the periodic boundary conditions, together with $p(0)=p(T)$ and $\overline{e}=0$, it follows $\overline{u}=0$.

Multiplying by $u$ the equation in \eqref{pbepsilon1} and integrating by parts over $[0,T]$, we obtain
$$-\int_0^T \phi_{p(t)}(u'(t))u'(t)dt-\lambda\int_0^T\Theta'(u(t))u'(t)dt-\lambda\varepsilon\int_0^T u^2(t)dt=\lambda\int_0^Te(t)u(t)dt,$$
or,
$$\int_0^T |u'(t)|^{p(t)}dt+\lambda\int_0^T\left(\Theta(u(t))\right)'dt+\lambda\varepsilon\int_0^T u^2(t)dt=-\lambda\int_0^T(\overline{e}+\tilde{e}(t))u(t)dt,$$
which gives
$$\int_0^T |u'(t)|^{p(t)}dt+\lambda\varepsilon\int_0^T u^2(t)dt=-\lambda\int_0^T\tilde{e}(t)u(t)dt\leq\int_0^T|\tilde{e}(t)u(t)|dt\leq\|\tilde{e}\|_{L^1}\|u\|_\infty.$$
Hence, using Sobolev inequality \eqref{sobolev}, we derive
\begin{equation*}\label{estim}
    \int_0^T |u'(t)|^{p(t)}dt\leq\|\tilde{e}\|_{L^1}\|u'\|_{L^1},
\end{equation*}
which yields
\begin{equation}\label{estim1}
\int\limits_{[|u'(t)|>1]}|u'(t)|^{p_-}dt \leq \int\limits_{[|u'(t)|>1]}|u'(t)|^{p(t)}dt \leq\|\tilde{e}\|_{L^1}\|u'\|_{L^1}.
\end{equation}
On the other hand, as
\begin{eqnarray*}
\int\limits_{[|u'(t)|\leq 1]}|u'(t)|^{p_-}dt \leq \int\limits_{[|u'(t)|\leq 1]}dt  \leq T
\end{eqnarray*}
and using \eqref{estim1} and H\"older inequality, we obtain
\begin{eqnarray*}
\int_0^T |u'(t)|^{p_-}dt &\leq& T + \|\tilde{e}\|_{L^1}\|u'\|_{L^1} \\
&\leq&  T + \|\tilde{e}\|_{L^1}T^{\frac{p_- - 1}{p_-}}\left(\int_0^T |u'(t)|^{p_-}dt \right)^{\frac{1}{p_-}}.
\end{eqnarray*}
Consequently,
 \begin{equation}\label{estim2}
T^{\frac{p_-}{p_- - 1}}\|u'\|_{L^1} \leq \|u'\|_{L^{p_-}} \leq R_1,
\end{equation}
where $R_1 := R_1(T,\tilde e,p_-)$ is the (positive) solution of the equation
$$x^{p_-} - \|\tilde{e}\|_{L^1}T^{\frac{p_- - 1}{p_-}}x - T = 0.$$
From \eqref{sobolev} and \eqref{estim2}, one has that
\begin{equation}\label{jestu}
\|u\|_\infty\leq R_2,
\end{equation}
with $R_2 := T^{\frac{p_- -1}{p_-}}R_1$.

Next, as $u$ is a solution of \eqref{pbepsilon1}, and since Rolle theorem implies the existence of  $\tau \in (0,T)$ such that $u'(\tau) = 0$, we have, for all $t \in [0,T]$,
\begin{eqnarray*}
|\phi_{p(t)}(u'(t))| &=& \lambda \left |\int_\tau^t  [e(s) - (\Theta'(u(s)))' + \varepsilon u(s)]ds \right| \\
&\leq& \|e \|_{L_1} + R_2 \max_{|v| \leq R_2}|\theta(v)| + \varepsilon^* R_2 T\leq  R_3,
\end{eqnarray*}
where
$$R_3 := R_3(\varepsilon^*) = \max\{\|e\|_{L_1} + R_2 \max_{|v| \leq R_2}|\theta(v)| + \varepsilon^* R_2T, 1\}.$$
Hence, for all $t \in [0,T]$,
\begin{equation}\label{jestudot}
|u'(t)| \leq R_3^{\frac{1}{p(t)-1}} \leq R_3^{\frac{1}{p_- - 1}} := R_4(\varepsilon^*).
\end{equation}
This, together with \eqref{jestu}, shows that $\|u\|_1<R:=R_2+R_4(\varepsilon^*)+1$ and condition $(H_1)$ in Theorem \ref{thHMMT} is fulfilled with $\Omega:=B(0,R)$.

To check the remaining conditions of Theorem \ref{thHMMT}, we see that
$$\mathcal{L}(a):=\frac{1}{T}\int_0^T(e(t)+\varepsilon a)dt=\varepsilon a$$
and $a=0$ is the unique solution of $\mathcal{L}(a)=0$. As $\partial \Omega \cap \mathbb{R}=\{ -R, R \}$, clearly $(H_2)$  is satisfied. Also,  $deg_B[\mathcal{L},\Omega\cap\mathbb{R},0]=1$, hence $(H_3)$ holds true.
The proof is complete.
\end{proof}

\begin{remark}\label{remarkDir} \emph{For a result in line with Theorem \ref{thepsilon}, but for a Li\'{e}nard system with Dirichlet boundary conditions we refer the reader to \cite[Theorem 5.1]{[HMMT1]}.}
\end{remark}

Let $\ell:[0,T]\times\mathbb{R}^2\to\mathbb{R}$ be a Carath\'{e}odory function. We denote by $N_\ell:C_T^1\to L^1$ the Nemytskii operator
associated to $\ell$; recall, this is
$$N_\ell(v)(t)=\ell(t,v(t),v'(t)), \quad (v\in C_T^1, \, t\in [0,T])$$
and it is known that $N_\ell$ is continuous and sends bounded sets from $C_T^1$  into equi-integrable sets in $L^1$. Also, defining
$$P:C_T^1\to C_T^1,\ \ v\mapsto v(0) \, \mbox{ and }\, Q:L^1\to L^1,\ \ h\mapsto\frac{1}{T}\int_0^T h(t)dt,$$
we have that $P$ is completely continuous and $Q$ is continuous and maps equi-integrable sets in $L^1$ into relatively compact sets in $C_T^1$.
\medskip

\begin{lemma}\label{thpfix} If $\ell:[0,T]\times\mathbb{R}^2\to\mathbb{R}$ is Carath\'{e}odory, then
there exists a continuous operator $\mathcal{K}:L^1\to C_T^1$ that sends equi-integrable sets in $L^1$ into relatively compact sets in $C_T^1$, such that $u\in C^1_T$ is a solution of problem \eqref{pbell} if and only if it is a fixed point of the completely continuous mapping
$$\mathcal{G}_\ell:C_T^1\to C_T^1,\ \ v\mapsto Pv+QN_\ell(v)+\mathcal{K}N_\ell(v).$$
\end{lemma}
\begin{proof} See Section 3 in \cite{[HMMT]}.
\end{proof}

The following result extends the existence part of Theorem \ref{thepsilon} to the case where $\varepsilon = 0$.

\begin{theorem}\label{thoper}
If $\theta:\mathbb{R}\to\mathbb{R}$ is continuous and $e\in L^1$ satisfies $\overline{e}=0$, then problem
\begin{equation}\label{pboper}
	\left \{
            \begin{array}{ll}
		      (\phi_{p(t)}(u'))'+\theta(u)u'=e(t)\quad \mbox{ in }[0,T],\\
		      u(0)-u(T)=0=u'(0)-u'(T)
			\end{array}
\right.
\end{equation}
has at least one solution $u$ with $\overline{u}=0$.
\end{theorem}
\begin{proof}
Let $\ell_\infty, \, \ell_n:[0,T]\times\mathbb{R}^2\to\mathbb{R}$ be defined by
$$\ell_\infty(t, x, y)=e(t)-\theta(x)y,\quad \ell_n=\ell_\infty(t, x, y)+\frac{1}{n}x,\quad (n\in\mathbb{N},\ (t,x,y)\in[0,T]\times\mathbb{R}^2).$$
From Theorem \ref{thepsilon} there exists $R=R(1)$ such that, for each $n\in\mathbb{N}$, problem
\begin{equation*}\label{pbelln}
	\left \{
            \begin{array}{ll}
		      (\phi_{p(t)}(u'))'=\ell_n(t,u,u')\quad \mbox{ in }[0,T],\\
		      u(0)-u(T)=0=u'(0)-u'(T)
			\end{array}
\right.
\end{equation*}
has at least one solution $u_n$ in $\overline{B}(0,R)$ with $\overline{u}_n=0$. For such a solution $u_n$, by Lemma \ref{thpfix}, one has
\begin{equation}\label{unG}
    u_n=\mathcal{G}_{\ell_n}(u_n)\in\mathcal{G}_{\ell_n}(\overline{B}(0,R))\ \ (n\in\mathbb{N}).
\end{equation}

Since $N_{\ell_\infty}(\overline{B}(0,R))$ is an equi-integrable set in $L^1$, there exists $\eta\in L^1$ such that $\left|N_{\ell_\infty}(v)\right|\leq\eta$ for all $v\in \overline{B}(0,R)$. It follows that $\left|N_{\ell_n}(v)\right|\leq\eta+R,$ for all $n\in\mathbb{N}$ and $v\in \overline{B}(0,R)$, which implies that the set
$$\mathcal{E}:=\bigcup_{n\in\mathbb{N}}N_{\ell_n}(\overline{B}(0,R))$$
is equi-integrable in $L^1$. Then, using
\begin{eqnarray*}
u_n \in \mathcal{G}_{\ell_n}(\overline{B}(0,R)) &\subset& P(\overline{B}(0,R))+ Q N_{\ell_n}(\overline{B}(0,R)) +\mathcal{K}N_{\ell_n}(\overline{B}(0,R))\\
&\subset& P(\overline{B}(0,R))+Q(\mathcal{E})+\mathcal{K}(\mathcal{E})=:\widehat{\mathcal{E}}, \quad (n\in \mathbb{N})
\end{eqnarray*}
and that $\widehat{\mathcal{E}}$ is relatively compact in $C^1_T$, we get the existence of a subsequence $\{u_{n_k}\}\subset\{u_n\}$ and an $u\in C_T^1$ so that $u_{n_k}\to u$ in $C_T^1$, with $k\to\infty$.
As
$$N_{\ell_{n_k}}(u_{n_k})=N_{\ell_\infty}(u_{n_k})+\frac{1}{n_k}u_{n_k},$$
we infer that
$$N_{\ell_{n_k}}(u_{n_k}) \to N_{\ell_\infty}(u), \; \mbox{ as } k \to \infty.$$
Then, from
$$u_{n_k}=\mathcal{G}_{\ell_{n_k}}(u_{n_k})=Pu_{n_k}+Q\left(N_{\ell_{n_k}}(u_{n_k})\right)+\mathcal{K}\left(N_{\ell_{n_k}}(u_{n_k})\right), \quad (k\in \mathbb{N}),$$
it follows
$$u=Pu+QN_{\ell_\infty}(u)+\mathcal{K}N_{\ell_\infty}(u)=\mathcal{G}_{\ell_\infty}(u),$$
i.e. $u$ is a fixed point of $\mathcal{G}_{\ell_\infty}$, meaning that  $u$ is a solution for problem \eqref{pboper} -- by Lemma \ref{thpfix}. The fact that $\overline{u}=0$ is straightforward using that $\overline{u}_{n_k}=0$, ($k\in \mathbb{N}$),  and the proof is complete.
\end{proof}

\begin{proposition}\label{lemaf0}
Let $e\in L^1$ be with $\overline{e}=0$. Then, for any $\theta:\mathbb{R}\to\mathbb{R}$ continuous and $c\in\mathbb{R}$, problem
\begin{equation}\label{pbc}
	\left \{
            \begin{array}{ll}
		      (\phi_{p(t)}(u'))'+\theta(c+u)u'=e(t)\quad \mbox{ in }[0,T],\\
		      u(0)-u(T)=0=u'(0)-u'(T)
			\end{array}
\right.
\end{equation}
has at least one solution $u$ with $\overline{u}=0$ and, for such a solution $u$, it holds
\begin{equation}\label{inegK}
    \|u\|_\infty\leq K,
\end{equation}
where $K=K(T,p_-,e)$ is a constant independent of $\theta$, $c$ and $u$.
\end{proposition}
\begin{proof}
The existence of a solution $u$ for \eqref{pbc} with $\overline{u}=0$ follows from Theorem \ref{thoper}. To prove the existence of $K$ in \eqref{inegK}, let $c_0=c_0(T,p_-)$ be a constant so that
\begin{equation}\label{inegc0}
  \int_0^T|v(t)|^{p_-}\ dt\leq c_0 \int_0^T|v'(t)|^{p_-}\ dt,\quad (v\in W^{1,p_-}, \, \overline{v}=0);
\end{equation}
such a constant is known to exist by Poincar\'{e}-Wirtinger inequality. Here, we have denoted by $W^{1,p_-}$ the Sobolev space $W^{1,p_-}([0,T], \mathbb{R})$, which is considered with its usual norm $\|v\|_{ W^{1,p_-}}=\left ( \|v\|_{ L^{p_-}}^{p_-}+\|v'\|_{ L^{p_-}}^{p_-}\right )^{1/p_{-}}$.  By the continuity of the embedding $W^{1,p_-}\subset C$, there is another constant $c_1=c_1(T,p_-)$ so that
\begin{equation}\label{inegc1}
    \|v\|_\infty\leq c_1 \|v\|_{ W^{1,p_-}}, \quad (v\in W^{1,p_-}).
\end{equation}

Multiplying the equation in \eqref{pbc} by $u$ and integrating over $[0, T]$, we deduce
$$-\int_0^T|u'(t)|^{p(t)}dt=\int_0^T e(t)u(t)dt,$$
whence
$$\int_0^T|u'(t)|^{p(t)}dt\leq\|e\|_{L^1}\|u\|_\infty.$$
Using this, we estimate
\begin{eqnarray}\label{inequal}
  \int_0^T|u'(t)|^{p_-}\ dt&=&\int\limits_{[|u'(t)|\leq 1]}|u'(t)|^{p_-}\ dt+\int\limits_{[|u'(t)|>1]}|u'(t)|^{p_-}\ dt\nonumber\\
  &\leq& T+\int_0^T|u'(t)|^{p(t)}dt\leq T+\|e\|_{L^1}\|u\|_\infty.
\end{eqnarray}
Then, from \eqref{inegc0} with $v=u$, and \eqref{inequal}, it follows
\begin{equation}\label{normw1p}
    \|u\|_{ W^{1,p_-}}^{p_-}=\int_0^T|u(t)|^{p_-}\ dt+\int_0^T|u'(t)|^{p_-}\ dt \leq(c_0+1)(T+\|e\|_{L^1}\|u\|_\infty)
\end{equation}
and from \eqref{inegc1} with $v=u$ and \eqref{normw1p}, we obtain
$$\|u\|_\infty^{p_-}\leq c_1^{p_-}(c_0+1)(T+\|e\|_{L^1}\|u\|_\infty),$$
which implies the existence of a constant $K=K(T,p_-,e)>0$ such that \eqref{inegK} holds true.
\end{proof}

\section{Lower and upper solutions}\label{sectiunea2}

It is well known that the method of lower and upper solutions is an efficient tool in providing the existence and localization of solutions of various elliptic boundary value problems (see e.g. the survey paper \cite{[CoHa]} and the monograph \cite{CoHa1}).
This will be a key ingredient in our approach for problem \eqref{pb1}.

\begin{definition}\label{lusol}
\emph{A \textit{lower solution} (resp. \textit{upper solution}) of problem \eqref{pb1} is a $C^1$-function $\alpha:[0,T]\to(0,\infty)$ with $\phi_{p(\cdot )}(\alpha'( \cdot ))\in C^1$, $\alpha(0)=\alpha(T)$, $\alpha'(0)\geq\alpha'(T)$ (resp. $\beta:[0,T]\to(0,\infty)$ with $\phi_{p(\cdot )}(\beta'( \cdot ))\in C^1$, $\beta(0)=\beta(T)$, $\beta'(0)\leq\beta'(T)$) and
$$(\phi_{p(t)}(\alpha'))'+f(\alpha)\alpha'+g(\alpha)\geq h(t),\quad (t\in[0,T]),$$
$$\left(\mbox{resp.}\ (\phi_{p(t)}(\beta'))'+f(\beta)\beta'+g(\beta)\leq h(t),\quad (t\in [0,T])\right).$$}
\end{definition}
\smallskip

The following lemma can be proved by elementary arguments.
\begin{lemma}\label{elemlema}
Let $a,b\in\mathbb{R}$ with $a<b$ and $\varphi:[0,T]\to\mathbb{R}$ be a continuous function. If there are $\tau_0, \tau_1\in [0,T]$ such that $\varphi(\tau_0)\leq a$ and $\varphi(\tau_1) \geq b$, then there exist $t_0, t_1\in [0,T]$ with $\varphi(t_0)=a$, $\varphi(t_1)=b$ and
$$\mbox{if } t_0<t_1, \mbox{ then }\varphi(t)\in[a,b],\qquad  (t\in[t_0,t_1]),$$
respectively,
$$\mbox{if } t_1<t_0, \mbox{ then }\varphi(t)\in[a,b],\qquad  (t\in[t_1,t_0]).$$
\end{lemma}
\smallskip

\begin{theorem}\label{lemalu}
If problem \eqref{pb1} has a lower solution $\alpha$ and an upper solution $\beta$ such that $\alpha(t)\leq\beta(t)$ for all $t\in[0,T]$,  then it has a solution $u$ with $\alpha(t)\leq u(t)\leq\beta(t)$ for all $t\in[0,T]$.
\end{theorem}

\begin{proof} Let $\gamma:[0,T]\times\mathbb{R}\to(0,\infty)$ be the continuous function given by
$$\gamma(t,x):=\left \{
            \begin{array}{lll}
		      \beta(t),\quad &\mbox{ if}& x>\beta(t),\\
		      x,\quad\ &\mbox{ if}& \alpha(t)\leq x\leq\beta(t),\\
              \alpha(t),\quad &\mbox{ if}& x<\alpha(t)
			\end{array}
\right.$$
and define $\ell^{*}:[0,T]\times\mathbb{R}^2\to\mathbb{R}$ by $$\ell^*(t,x,y)=h(t)-g(\gamma(t,x))-f(\gamma(t,x))y+x-\gamma(t,x).$$ Note that, as $\ell^{*}$ is continuous, below the notion of solution will be understood in classical sense (see Remark \ref{remark1}).  We introduce the modified problem
\begin{equation}\label{pbmod}
	\left \{
            \begin{array}{ll}
		      (\phi_{p(t)}(u'))'=\ell^*(t,u,u')\quad \mbox{ in }[0,T],\\
		      u(0)-u(T)=0=u'(0)-u'(T)
			\end{array}
\right.
\end{equation}
and, in order to apply Theorem \ref{thHMMT}, for $\lambda\in(0,1)$, we consider the homotopy problem associated with \eqref{pbmod}
\begin{equation}\label{pbmodlambda}
	\left \{
            \begin{array}{ll}
		      (\phi_{p(t)}(u'))'=\lambda\ell^*(t,u,u')\quad \mbox{ in }[0,T],\\
		      u(0)-u(T)=0=u'(0)-u'(T).
			\end{array}
\right.
\end{equation}

Next, let $u \equiv u_{\lambda}$ be a solution of problem \eqref{pbmodlambda} for some $\lambda\in(0,1)$. We choose $M_1>\|\beta\|_\infty$ a constant such that
\begin{equation}\label{M1}
    h(t)-g(\beta(t))+M_1-\beta(t)>0\ \mbox{ and }\ h(t)-g(\alpha(t))-M_1-\alpha(t)<0 \quad (t\in[0,T])
\end{equation}
and we show that $\|u\|_\infty<M_1$. Otherwise, suppose that there exists $t_0\in[0,T]$ such that $u(t_0)=\max_{t\in[0,T]}u(t)\geq M_1$ or $u(t_0)=\min_{t\in[0,T]}u(t)\leq -M_1$.
If $t_0\in(0,T)$ and $u(t_0)\geq M_1$, then $u'(t_0)=0$ and since $u$ is a solution of \eqref{pbmodlambda}, on account of $M_1>\|\beta\|_\infty$ and \eqref{M1}, one has
\begin{eqnarray*}
  (\phi_{p(t)}(u'))'_{t=t_0} &=& \lambda\ell^*(t_0,u(t_0),0)=\lambda(h(t_0)-g(\beta(t_0))+u(t_0)-\beta(t_0)) \\
  &\geq& \lambda(h(t_0)-g(\beta(t_0))+M_1-\beta(t_0))>0.
\end{eqnarray*}
So, there exists a positive constant $\delta$ such that $\phi_{p(t)}(u'(t))$ is strictly increasing on $(t_0,t_0+\delta)$. Thus, $$\phi_{p(t)}(u'(t))>\phi_{p(t_0)}(u'(t_0))=0,\quad (t\in(t_0,t_0+\delta)),$$ which shows that $u'(t)>0,\ t\in (t_0,t_0+\delta)$. Hence, $u$ is strictly increasing on $(t_0,t_0+\delta)$, contradicting $u(t_0)=\max_{t\in[0,T]}u(t)$. In a similar way, if $t_0\in(0,T)$ and $u(t_0)\leq-M_1$, since $-M_1<\alpha(t)$, for all $t\in[0,T]$ and using the second inequality in \eqref{M1}, we obtain
$$(\phi_{p(t)}(u'))'_{t=t_0} = \lambda\ell^*(t_0,u(t_0),0)\leq\lambda(h(t_0)-g(\alpha(t_0))-M_1-\alpha(t_0))<0,$$
which implies that there is $\delta>0$ such that $u$ is strictly decreasing on $(t_0,t_0+\delta)$, i.e. again a contradiction with $u(t_0)=\min_{t\in[0,T]}u(t)$.

If $t_0=0$ or $t_0=T$ and $u(0)=u(T) \geq M_1$, then $u'(0) \leq 0$ and $u'(T) \geq 0$ implying that $u'(0)=0$ and arguing as above we get a contradiction. The case when $t_0=0$ or $t_0=T$ under the assumption $u(0)=u(T) \leq -M_1$ can be treated by similar arguments.

Next, we claim that there exists a constant $M_2>0$, such that $\|u'\|_\infty<M_2$. To prove this, observe first that
\begin{equation}\label{nagumo}
    |(\phi_{p(t)}(u'))'|\leq|\ell^*(t,u,u')|\leq\|h\|_{\infty}+c_1+c_2|u'|+2M_1\leq c_3(1+|u'|),\quad (t\in[0,T]),
\end{equation}
where $c_1$, $c_2$ and $c_3$ are positive constants depending on $M_1$. Since $u(0)=u(T)$, there is some $\tau_0\in[0,T]$ such that $u'(\tau_0)=0$. Then, obviously one has
\begin{equation}\label{utau0}
    0=|u'(\tau_0)|^{p(\tau_0)-1}<\left(\frac{2M_1}{T}+1\right)^{p^+-1}.
\end{equation}
\noindent We choose the constant $M_2>1$ such that $$\frac{M_2^{p_--1}-\left(\displaystyle\frac{2M_1}{T}+1\right)^{p^+-1}}{c_3}>T+2M_1$$
and suppose by contradiction that there is a point $\tau_1\in [0,T]$ such that $|u'(\tau_1)|\geq M_2$, implying that
\begin{equation}\label{utau1}
    |u'(\tau_1)|^{p(\tau_1)-1}\geq M_2^{p_--1}.
\end{equation}
In view of \eqref{utau0}, \eqref{utau1} and Lemma \ref{elemlema}, there exist $t_0,t_1\in[0,T]$ with $$|u'(t_0)|^{p(t_0)-1}=\left(\displaystyle\frac{2M_1}{T}+1\right)^{p^+-1},\ \ |u'(t_1)|^{p(t_1)-1}=M_2^{p_--1}$$ and
$$\mbox{if } t_0<t_1, \mbox{ then } |u'(t)|^{p(t)-1}\in\left[\left(\displaystyle\frac{2M_1}{T}+1\right)^{p^+-1},M_2^{p_--1}\right],\ \ (t\in[t_0,t_1]),$$
respectively,
$$\mbox{if } t_1<t_0, \mbox{ then } |u'(t)|^{p(t)-1}\in\left[\left(\displaystyle\frac{2M_1}{T}+1\right)^{p^+-1},M_2^{p_--1}\right],\ \  (t\in[t_1,t_0]).$$
We discuss the case $t_0<t_1$; similar arguments work when $t_1<t_0$. Using the first change of variables formula, we have
\begin{eqnarray}\label{estim0}
  T+2M_1 &<& \frac{M_2^{p_--1}-\left(\displaystyle\frac{2M_1}{T}+1\right)^{p^+-1}}{c_3} =\int\limits_{\left(\frac{2M_1}{T}+1\right)^{p^+-1}}^{M_2^{p_--1}}\frac{1}{c_3}\ ds\nonumber\\
  &=& \int\limits_{|u'(t_0)|^{p(t_0)-1}}^{|u'(t_1)|^{p(t_1)-1}}\frac{1}{c_3}\ ds=\int_{t_0}^{t_1}\frac{(|u'(t)|^{p(t)-1})'}{c_3}\ dt.
\end{eqnarray}
Note that $u'$ does not vanish in $[t_0,t_1]$. If $u'(t)$ is strictly positive on $[t_0,t_1]$, then on account of \eqref{estim0} and \eqref{nagumo}, we obtain
\begin{eqnarray*}
  T+2M_1 &<& \int_{t_0}^{t_1}\frac{(\phi_{p(t)}(u'))'}{c_3}\ dt=\int_{t_0}^{t_1}\frac{\lambda\ell^*(t,u,u')}{c_3}\ dt \\
  &\leq& \int_{t_0}^{t_1}(1+u'(t)) dt\leq T+|u(t_1)-u(t_0)|<T+2M_1,
\end{eqnarray*}
i.e., a contradiction. Similarly, if $u'(t)$ is strictly negative on $[t_0,t_1]$, one gets
\begin{eqnarray*}
  T+2M_1 &<& \int_{t_0}^{t_1}\frac{-(\phi_{p(t)}(u'))'}{c_3}\ dt=\int_{t_0}^{t_1}\frac{-\lambda\ell^*(t,u,u')}{c_3}\ dt \\
  &\leq& \int_{t_0}^{t_1}(1-u'(t)) dt\leq T+|u(t_0)-u(t_1)|<T+2M_1,
\end{eqnarray*}
 a contradiction, again.
\smallskip

Hence, $\|u\|_1<M_0$, where $M_0:=M_1+M_2$, and we shall apply Theorem \ref{thHMMT} with $\Omega:=B(0,M_0)$ and $\ell=\ell^*$.
Clearly,  $(H_1)$ in Theorem \ref{thHMMT} is satisfied.
From \eqref{M1}, we have that $\ell^*(t,-M_0,0)<0$ and $\ell^*(t,M_0,0)>0$ for all $t\in [0,T]$, which yield
$$\mathcal{L}(-M_0)=\frac{1}{T}\int_0^T\ell^*(t,-M_0,0)dt<0,\ \ \ \mathcal{L}(M_0)=\frac{1}{T}\int_0^T\ell^*(t,M_0,0)dt>0$$
and, as $\partial \Omega \cap \mathbb{R}=\{ -M_0, M_0 \}$, hypothesis $(H_2)$  is fulfilled. Finally,  $deg_B[\mathcal{L},\Omega\cap\mathbb{R},0]=1$, hence $(H_3)$ also holds true. Therefore, problem \eqref{pbmod} has at least one solution in $\overline{\Omega}$.

We conclude the proof by showing that if $u$ is a solution of \eqref{pbmod}, then $\alpha(t)\leq u(t)\leq\beta(t)$ for all $t\in[0,T]$. Suppose by contradiction that there is some $t_0\in[0,T]$ such that $$\max_{t\in[0,T]}(\alpha(t)-u(t))=\alpha(t_0)-u(t_0)>0.$$
If $t_0\in(0,T)$, then $\alpha'(t_0)=u'(t_0)$ and there is a sequence $\{t_k\}$ in $(t_0, T)$ converging to $t_0$ such that $\alpha'(t_k)-u'(t_k)\leq0$. As $\phi_{p(t)}: \mathbb{R} \to \mathbb{R}$
is an increasing homeomorphism for all $t\in [0,T])$, this yields
$$\phi_{p(t_k)}(\alpha'(t_k))-\phi_{p(t_0)}(\alpha'(t_0))\leq\phi_{p(t_k)}(u'(t_k))-\phi_{p(t_0)}(u'(t_0)),$$
implying that
$$\left(\phi_{p(t)}(\alpha'(t))\right)_{t=t_0}'\leq\left(\phi_{p(t)}(u'(t))\right)_{t=t_0}'.$$
Thus, because $\alpha'(t_0)=u'(t_0)$ and $\alpha$ is a lower solution of \eqref{pb1}, we obtain
\begin{eqnarray*}
  \left(\phi_{p(t)}(\alpha'(t))\right)_{t=t_0}' &\leq& \left(\phi_{p(t)}(u'(t))\right)_{t=t_0}'=\ell^*(t_0,u(t_0),u'(t_0)) \\
  &=& h(t_0)-g(\alpha(t_0))-f(\alpha(t_0))\alpha'(t_0)+u(t_0)-\alpha(t_0)\\
  &<& h(t_0)-g(\alpha(t_0))-f(\alpha(t_0))\alpha'(t_0)\leq \left(\phi_{p(t)}(\alpha'(t))\right)_{t=t_0}',
\end{eqnarray*}
a contradiction. If $$\max_{t\in[0,T]}(\alpha(t)-u(t))=\alpha(0)-u(0)=\alpha(T)-u(T)>0,$$
then $\alpha'(0)-u'(0)\leq0$ and $\alpha'(T)-u'(T)\geq0$. Since $\alpha$ is a lower solution of \eqref{pb1}, one has $\alpha'(0)\geq\alpha'(T)$ and using this, we get that $\alpha'(0)-u'(0)=0=\alpha'(T)-u'(T)$. It follows  $\phi_{p(0)}(\alpha'(0))=\phi_{p(0)}(u'(0))$. On the other hand, $\max_{t\in[0,T]}(\alpha(t)-u(t))=\alpha(0)-u(0)>0$ implies, reasoning as above, that
$$\left(\phi_{p(t)}(\alpha'(t))\right)_{t=0}'\leq\left(\phi_{p(t)}(u'(t))\right)_{t=0}'.$$
Using this inequality, together with $\alpha'(0)=u'(0)$ and the fact that $\alpha$ is a lower solution of \eqref{pb1}, we again proceed as above to obtain a contradiction. Consequently, $\alpha(t)\leq u(t)$ for all $t\in[0,T]$. Analogously, using that $\beta$ is an upper solution of problem \eqref{pb1}, it follows that $u(t)\leq\beta(t)$ for all $t\in[0,T]$. The proof is now complete.
\end{proof}
\smallskip

\begin{remark}\label{rem2}\emph{The variant of Theorem \ref{lemalu} with an increasing and odd homeomorphism $\phi:\mathbb{R} \to \mathbb{R}$ instead of $\phi_{p(t)}$,  can be inferred as a particular case of \cite[Theorem 8.12]{[RaStTv]}.  }
\end{remark}

\section{Main result}\label{final}

We are now in position to give the main result concerning problem \eqref{pb1}.

\begin{theorem}\label{mainth}
If there exists a constant $\alpha>0$ with $g(\alpha) \geq \max_{[0,T]}h$ and it holds
\begin{equation}\label{conditieg}
\limsup_{x\to+\infty}g(x)<\overline{h},
\end{equation}
then problem \eqref{pb1} has at least one solution $u$ with $u(t)\geq\alpha$, for all $t\in[0,T]$.
\end{theorem}
\begin{proof}
Clearly, $\alpha$ is a lower solution of problem \eqref{pb1}. On the other hand, from assumption \eqref{conditieg}, there exists a constant $\delta>\alpha$ such that $g(x) <\overline{h}$, for all $x\geq\delta$. We define the continuous function $f_0:\mathbb{R}\to\mathbb{R}$ by
$$f_0(x)=\left\{
           \begin{array}{ll}
             f(0), & \hbox{if $x<0$,} \\
             f(x), & \hbox{if $x\geq 0$}
           \end{array}
         \right.
$$
and for $c\in\mathbb{R}$ let us consider the problem
\begin{equation}\label{pbcf0}
	\left \{
            \begin{array}{ll}
		      (\phi_{p(t)}(v'))'+f_0(c+v)v'=\tilde{h}(t)\quad \mbox{ in }[0,T],\\
		      v(0)-v(T)=0=v'(0)-v'(T).
			\end{array}
\right.
\end{equation}
which is of type \eqref{pbc}. Then, by Proposition \ref{lemaf0}, there is a constant $K=K(T,p_-,\tilde{h})>0$ so that any solution $v$ of \eqref{pbcf0} with $\overline{v}=0$ satisfies $\|v\|_\infty\leq K$. Note that a solution of \eqref{pbcf0} is a classical one (see Remark \ref{remark1}).

We choose $c_\delta=\delta+K$ and let $v_\delta$ with $\overline{v}_\delta=0$ be a solution of problem
\begin{equation*}
	\left \{
            \begin{array}{ll}
		      (\phi_{p(t)}(v'))'+f_0(c_\delta+v)v'=\tilde{h}(t)\quad \mbox{ in }[0,T],\\
		      v(0)-v(T)=0=v'(0)-v'(T).
			\end{array}
\right.
\end{equation*}
One has that $\beta(t):=c_\delta+v_\delta(t)\geq c_\delta-K=\delta$ and
$$(\phi_{p(t)}(\beta'(t)))'+f(\beta(t))\beta'(t)+g(\beta(t))$$
$$= (\phi_{p(t)}(v_\delta'(t)))'+f_0(c_\delta+v_\delta(t))v_\delta'(t)+g(c_\delta+v_\delta(t))<\tilde{h}(t)+\overline{h}=h(t)\quad (t\in[0,T]),$$
i.e. $\beta$ is an upper solution for \eqref{pb1}. Thus, as $\beta(t)\geq\alpha$ for all $t\in [0,T]$, by Theorem \ref{lemalu} there exists
 a solution $u$ of problem \eqref{pb1} with $\alpha\leq u(t)\leq\beta(t)$, for all $t\in[0,T]$,  and the proof is complete.
\end{proof}
\smallskip

\begin{remark}\label{remMJM1}\emph{The particular case of Theorem \ref{mainth} with $p\in (1, \infty)$ constant instead of  $p(t) \in C\left([0,T];(1,\infty)\right)$ was proved in \cite[Theorem 1]{[JMANS]} (also see \cite[Theorem 1]{[JMVJM]}). An extension of the result from \cite{[JMANS]} was provided in \cite[Theorem 2.3]{[RaT3]}, \cite[Theorem 8.21]{[RaStTv]}.
It is worth to notice that the classical case $p=2$ was obtained in \cite[Theorem 1]{[HaSa]} (see also \cite[Theorem 6.1]{[Maw93]}). }
\end{remark}
\medskip

The following result was first proved in \cite{[LaSo]} for $p=2$ and $f=0$ (see the proof of Theorem 2.1 therein).

\begin{corollary}\label{cormth}
If the continuous function $g:(0,\infty)\to (0,\infty)$ satisfies
\begin{equation*}
g(x)\to+\infty\ \ \mbox{as}\  x\to0+ \; \mbox{ and } \; g(x)\to0\  \mbox{as}\  x\to+\infty,
\end{equation*}
then problem \eqref{pb1} has at least one solution if and only if $\overline{h}>0$.
\end{corollary}
\begin{proof} If $\overline{h}>0$ then the existence of a solution follows by Theorem \ref{mainth}, while the converse implication is obtained by integrating the differential equation in \eqref{pb1} over $[0,T]$ and using the fact that $g>0$.
\end{proof}
\medskip

\begin{example}
{\em If $f:[0,\infty)\to\mathbb{R}$ is continuous, $h\in C$, $\sigma>0$ and $\mu>0$, then problem
\begin{equation*}
	\left \{
            \begin{array}{ll}
		      (\phi_{p(t)}(u'))'+f(u)u'+\displaystyle\frac{\sigma}{u^\mu}=h(t)\quad \mbox{ in }[0,T],\\
              \\
		      u(0)-u(T)=0=u'(0)-u'(T),
			\end{array}
\right.
\end{equation*}
has at least one solution if and only if $\overline{h}>0$.}
\end{example}


\begin{thebibliography}{99}
\bibitem{[CLT]} A. Cabada, A. Lomtatidze and M. Tvrd\'y, Periodic problem involving quasilinear differential operator and weak singularity, \textit{Adv. Nonlinear Stud.} {\bf 7}  (2007), 629--649.
\bibitem{[CoHa]} C. De Coster and P. Habets, The lower and upper solutions method for boundary value problems, \textit{Handbook of Differential Equations}, Elsevier/North-Holland, Amsterdam, 2004, 69--160.
\bibitem{CoHa1} C. De Coster and P. Habets, \textit{Two-Point Boundary Value Problems : Lower and Upper Solutions}, Mathematics in Science and Engineering, vol. 205, Elsevier, Amsterdam, 2006.
\bibitem{[HMMT]} M. Garc\'{i}a-Huidobro, R. Man\'{a}sevich, J. Mawhin and S. Tanaka, Periodic solutions for nonlinear systems of ODE's with generalized variable exponents operators, \textit{J. Differential Equations} {\bf 388} (2024), 34--58.
\bibitem{[HMMT1]} M. Garc\'{i}a-Huidobro, R. Man\'{a}sevich, J. Mawhin and S. Tanaka, Two point boundary value problems for ordinary differential systems with generalized variable exponents operators, \textit{Nonlinear Anal. Real World Appl.} {\bf 81} (2024), 13 p.
\bibitem{[HaSa]} P. Habets and L. Sanchez, Periodic solutions of some Li\'{e}nard equations with singularities, \textit{Proc. Amer. Math. Soc.} {\bf 109} (1990), 1035--1044.
\bibitem{[JMANS]} P. Jebelean and J. Mawhin, Periodic solutions of singular nonlinear perturbations of the ordinary $p$-Laplacian, \textit{ Adv. Nonlinear Stud.} {\bf 2} (2002), 299--312.
\bibitem{[JMVJM]} P. Jebelean and J. Mawhin, Periodic solutions of forced dissipative  $p$--Li\'{e}nard equations with singularities, \textit{ Vietnam J. Math.} {\bf 32} (2004), 97--103.
\bibitem{[LaC]} D. Lan and W. Chen, Periodic solutions for $p$-Laplacian differential equation with singular forces of attractive type, \textit{IAENG International J. Appi. Math.} {\bf 48} (2018), No. 1, 5 p.
\bibitem{[LaSo]} A.C. Lazer and S. Solimini, On periodic solutions of nonlinear differential equations with singularities, \textit{Proc. Amer. Math. Soc.} {\bf 99} (1987), 109--114.
\bibitem{[LiZ]} C. Liu and Y. Zhong, Infinitely many periodic solutions for ordinary $p(t)$-Laplacian differential systems, \textit{Electronic Research Archive ERA} {\bf 30} (2022), No. 5, 1653--1667.
\bibitem{[MaM]} R. Man\'asevich and J. Mawhin, Periodic solutions for nonlinear systems with $p$-Laplacian like operators, {\em J. Differential Equations} {\bf 145} (1998), 367--393.
\bibitem{[Maw93]} J. Mawhin, Topological degree and boundary value problems for nonlinear differential equations, \textit{Topological Methods in Ordinary Differential Differential Equations} (CIME, Montecatini Terme, 1991), M. Furi, P. Zecca eds., Lecture Notes Math., vol. 1537, Springer, Berlin, 1993, 74-142.
\bibitem{[Maw01]} J. Mawhin, Periodic Solutions of Systems with $p$-Laplacian-like Operators, \textit{Nonlinear Analysis and its Applications to Differential Equations}, M.R. Grossinho, M. Ramos, C. Rebelo, C., L. Sanchez eds., Progress in Nonlinear Differential Equations and Their Applications, vol. 43, Birkh\"{a}user, Boston, 2001, 37--63.
\bibitem{[MaT2]} J. Mawhin and H.B. Thompson,  Nagumo conditions and second-order quasilinear equations with compatible nonlinear functional boundary conditions, {\em Rocky Mountain Math. J.} {\bf 41} (2011), 573--596.
\bibitem {[RaT1]} I. Rach\r{u}nkov\'{a} and M. Tvrd\'{y}, Second-order periodic problem with $\phi$-Laplacian and impulses, \textit{Nonlinear Anal.} {\bf 63} (2005), 257--266.
\bibitem{[RaT2]} I. Rach\r{u}nkov\'{a} and M. Tvrd\'{y}, Periodic problems with $\phi$-Laplacian involving non-ordered lower and upper functions, \textit{Fixed Point Theory} {\bf 6} (2005), 99--112.
\bibitem{[RaT3]} I. Rach\r{u}nkov\'{a} and M. Tvrd\'{y}, Periodic singular problem with quasilinear differential operator, \textit{Math. Bohem.} {\bf 131} (3) (2006), 321--336.
\bibitem{[RTV]} I. Rach\r{u}nkov\'{a}, M. Tvrd\'{y} and I.Vrko\v{c},  Resonance and multiplicity in periodic boundary value problems with singularity, \textit{Math. Bohem.} {\bf 128} (2003), 45--70.
\bibitem{[RaStTv]} I. Rach\r{u}nkov\'{a}, S. Stan\v{e}k and M. Tvrd\'{y}, Solvability of Nonlinear Singular Problems for Ordinary Differential Equations, \textit{Contemp. Math. Appl.} {\bf 5}, Hindawi Publishing Corporation,  New York, 2008.
\bibitem{[WaY]} X.J. Wang and R. Yuan, Existence of periodic solutions for $p(t)$-Laplacian systems, \textit{Nonlinear Anal.} {\bf  70} (2009), 866--880.
\bibitem{[XiC1]} Y. Xin and Z. Cheng, Positive periodic solution of $p$-Laplacian Li\'enard type differential equation with singularity and deviating argument, \textit{Adv. Difference Equ.}, 2016, Paper No. 41, 11 p.
\bibitem{[XiC2]} Y. Xin and Z. Cheng, Study on a kind of $\phi$--Laplacian Li\'enard equation with attractive and repulsive singularities, \textit{J. Inequal. Appl.}, 2017, Paper No. 180, 12 p.
\bibitem{[XiC3]} Y. Xin and Z. Cheng, Multiple results to $\varphi$-Laplacian singular Li\'enard equation and applications, \textit{J. Fixed Point Theory Appl.} {\bf 23} (2021), Paper No. 21, 21 p.
\bibitem{[XHC]} Y. Xin, X.  Han and Z. Cheng, Existence and uniqueness of positive periodic solution for $p$-Laplacian Li\'enard equation, \textit{Bound. Value Probl.}, 2014, Paper No. 244, 11 p.
\bibitem{[XiL]} Y. Xin and H. Liu, Singularities of attractive and repulsive type for $p$-Laplacian generalized Li\'enard equation, \textit{Adv. Difference Equ.}, 2018, Paper No. 471, 21 p.
\bibitem{[ZhC]} L. Zhang and Y. Chen, Existence of periodic solutions of $p(t)$-Laplacian systems, \textit{Bull. Malays. Math. Sci. Soc.} (2) {\bf 35}  (2012), No. 1, 25--38.
\bibitem{[ZhT]} L. Zhang and  X.H. Tang, Existence and multiplicity of periodic solutions for some second order differential systems with $p(t)$-Laplacian, \textit{Math. Slovaca} {\bf 63} (2013), 1269--1290.
\bibitem{[Zhe]} F. Zheng, Periodic wave solutions of a non-Newtonian filtration equation with an indefinite singularity, \textit{AIMS Math.} {\bf 6} (2020), No. 2, 1209--1222.
\bibitem{[ZDD]} T. Zhou, B. Du and H. Du, Positive periodic solution for indefinite singular Li\'enard equation with $p$-Laplacian, \textit{Adv. Difference Equ.}, 2019, Paper No. 158, 17 p.
\bibitem{[Zhu]} Y. Zhu, Existence of positive periodic solutions for Li\'enard equation with a singularity of repulsive type, \textit{Bound. Value Probl.}, 2004, Paper No. 85, 12 p.
\end{thebibliography}
\end{document}